\documentclass{amsproc}

\usepackage{amscd,amsfonts,amssymb,amsmath,amsthm,latexsym}
\usepackage{graphicx}
\usepackage[all]{xy}

\newtheorem{theorem}{Theorem}[section]

\newtheorem{proposition}[theorem]{Proposition}
\newtheorem{conjecture}[theorem]{Conjecture}

\theoremstyle{definition}
\newtheorem{definition}[theorem]{Definition}
\newtheorem{example}[theorem]{Example}

\theoremstyle{remark}
\newtheorem{remark}[theorem]{Remark}

\numberwithin{equation}{section}

\newcommand{\suchthat}{\ | \ }
\newcommand{\marked}{\mathbb{M}}
\newcommand{\punct}{\mathbb{P}}
\newcommand{\surf}{(\Sigma,\marked)}
\newcommand{\qtau}{Q(\tau)}
\newcommand{\stau}{S(\tau)}
\newcommand{\qstau}{(\qtau,\stau)}
\newcommand{\qsigma}{Q(\sigma)}
\newcommand{\ssigma}{S(\sigma)}

\newcommand{\ka}{\mathbb{C}}
\newcommand{\complete}[1]{R\langle\hspace{-0.05cm}\langle #1\rangle\hspace{-0.05cm}\rangle}
\newcommand{\usual}[1]{R\langle #1\rangle}
\newcommand{\usualRQ}{\usual{Q}}
\newcommand{\completeRQ}{\complete{Q}}
\newcommand{\maxid}{\mathfrak{m}}
\newcommand{\jacob}[2]{\mathcal{P}(#1,#2)}
\newcommand{\jacobqs}{\jacob{Q}{S}}
\newcommand{\qstriv}[2]{(#1_{\operatorname{triv}},#2_{\operatorname{triv}})}
\newcommand{\qsred}[2]{(#1_{\operatorname{red}},#2_{\operatorname{red}})}
\newcommand{\Hom}{\operatorname{Hom}}
\newcommand{\Ginz}[2]{\widehat{\Gamma}(#1,#2)}
\newcommand{\bx}{\mathbf{x}}

\begin{document}

\title{On triangulations, quivers with potentials and mutations}

\author{Daniel Labardini-Fragoso}

\address{Instituto de Matem\'aticas, Universidad Nacional Aut\'onoma de M\'exico}

\email{labardini@matem.unam.mx}

\begin{abstract}
In this survey article we give a brief account of constructions and results concerning the quivers with potentials associated to triangulations of surfaces with marked points. Besides the fact that the mutations of these quivers with potentials are compatible with the flips of triangulations, we mention some recent results on the representation type of Jacobian algebras and the uniqueness of non-degenerate potentials. We also mention how the the quivers with potentials associated to triangulations give rise to CY2 and CY3 triangulated categories that have turned out to be useful in the subject of stability conditions and in theoretical physics.
\end{abstract}

\maketitle

\tableofcontents

\section*{Introduction}

Around 11 years ago, Fomin-Zelevinsky defined cluster algebras \emph{in an attempt to create an algebraic framework for dual canonical bases and total positivity in semisimple groups} (cf. \cite{FZ1}, from whose abstract the emphasized line was taken). Since then, cluster algebras have been found to 
possess interactions with a wide variety of areas, like Poisson geometry, integrable systems, Teichm\"{u}ller theory, Lie theory, representation theory of associative algebras, hyperbolic 3-manifolds, commutative and non-commutative algebraic geometry, mirror symmetry, KP solitons, and even with string theory in Physics.

Fundamental in the definition of cluster algebras is the notion of \emph{quiver mutation}, which is a combinatorial operation on quivers.
In a representation-theoretic approach to cluster algebras, Derksen-Weyman-Zelevinsky developed in \cite{DWZ1} a \emph{mutation theory of quivers with potentials}, which lifts quiver mutation from the combinatorial to the algebraic level. A quiver with potential (\emph{QP} for short) is a pair consisting of a quiver $Q$ and a potential $S$ on $Q$, that is, a linear combination of cycles of $Q$. The mutation theory of quivers with potentials ultimately leads to the notion of mutation of representations of QPs, thus providing a representation-theoretic interpretation for the combinatorial operation of quiver mutations.

On the other hand, a class of cluster algebras arising from triangulations of Riemann surfaces was introduced and systematically studied in \cite{FST} by Fomin-Shapiro-Thurston. These authors show that the elementary operation of \emph{flip} of arcs in triangulations can be interpreted as the operation of mutation inside the corresponding cluster algebra. In particular, they show that every triangulation $\tau$ of a Riemann surface with marked points has a naturally associated quiver $\qtau$, and that the flip of triangulations is reflected in the quiver level as quiver mutation.

In this survey article we describe a construction from \cite{L1} that associates a potential $\stau$ to each triangulation $\tau$, in such a way that the operation of flip is reflected at the level of QPs as the mutation of Derksen-Weyman-Zelevinsky. Then we state some results regarding the finite-dimensionality and the representation type of the Jacobian algebras of the QPs $\qstau$, as well as the uniqueness of non-degenerate potentials for the quivers $\qtau$. Finally, we mention a couple of applications that the QPs $\qstau$ have had in the subject of stability conditions and in theoretical physics.

The paper is divided in five sections. In Section \ref{sec:3-ops}, after recalling some elementary facts concerning (complete) path algebras of quivers (Subsection \ref{subsec:quivers}), we describe the combinatorial operation of quiver mutation (Subsection \ref{subsec:quiver-mutations}), then we give a quick overview of mutations of quivers with potentials (Subsection \ref{subsec:QP-mutations}), and close the section with a brief reminder of the setup of surfaces with marked points, their triangulations and flips of triangulations (Subsection \ref{subsec:triangulations}).

In Section \ref{sec:QP-of-triangulations}, we quickly say how to attach a quiver $\qtau$ to each triangulation $\tau$ of a surface with marked points and state the compatibility between flips and quiver mutations (Subsection \ref{subsec:quiver-of-triangulations}). In Subsection \ref{subsec:potential-of-triangulations} we lift the story to the level of QPs, that is, we describe a way to associate a QP $\qstau$ to each triangulation $\tau$, and state the compatibility between flips and QP-mutations.

In Section \ref{sec:Jacobian-algebras} we state results regarding the finite-dimensionality and the representation type of the Jacobian algebras of the QPs $\qstau$. It turns out that the Jacobian algebras $\jacob{\qtau}{\stau}$ are always finite-dimensional and tame. Also, the tori with exactly one marked point are the only surfaces whose triangulations have quivers that admit a non-degenerate potential with wild Jacobian algebra. Moreover, if $Q$ is a quiver admitting a non-degenerate potential with tame Jacobian algebra, then $Q$ is either the quiver of a triangulation, or mutation-equivalent to one of nine exceptional quivers.

In section \ref{sec:uniqueness} we mention some results on the uniqueness of non-degenerate potentials for the quivers $\qtau$. It turns out that most quivers arising from surfaces admit exactly one non-degenerate potential up to right-equivalence.

Finally, in Section \ref{sec:applications} we indicate how QPs give rise to triangulated Calabi-Yau categories, and mention how, via such categories, the QPs $\qstau$ have had applications in the subject of stability conditions and in theoretical physics.

\section{Three operations}\label{sec:3-ops}

\subsection{Quivers and path algebras}\label{subsec:quivers}

Recall that a \emph{quiver} is a finite graph with oriented edges, that is, a quadruple $Q=(Q_0,Q_1,t,h)$ consisting of a finite set of \emph{vertices} $Q_0$, a finite set of \emph{arrows}, and a pair of functions $t,h:Q_1\rightarrow Q_0$ that determine the \emph{tail} $t(\alpha)$ and the head $h(\alpha)$ of any given arrow $\alpha\in Q_1$. We write $\alpha:j\rightarrow i$ to indicate that $t(\alpha)=j$ and $h(\alpha)=i$. We will always assume that the quivers we work with are \emph{loop-free}, that is, that no arrow $\alpha$ satisfies $h(\alpha)=t(\alpha)$.

A \emph{path of length} $\ell>0$ on $Q$ is a sequence $a=\alpha_1\alpha_2\ldots \alpha_\ell$ of arrows with $t(\alpha_j)=h(\alpha_{j+1})$ for $j=1,\ldots,\ell-1$. We set $h(a)=h(\alpha_1)$ and $t(a)=t(\alpha_\ell)$. Positive-length paths are composed as functions, that is, if $a=\alpha_1\ldots \alpha_\ell$ and $b=\beta_1\ldots \beta_{\ell'}$ are paths with $h(b)=t(a)$, then the concatenation $ab$ is defined as the path $\alpha_1,\ldots \alpha_\ell\beta_1\ldots \beta_{\ell'}$, which starts at $t(ab)=t(\beta_{\ell'})$ and ends at $h(ab)=h(\alpha_1)$.

For each vertex $i\in Q_0$ we formally introduce a \emph{length-0 path} $e_i$. By $A^\ell$ we denote the $\ka$-vector space with basis the set of paths of length $\ell\geq 0$. We use the notations $R=A^0$ and $A=A^1$. Note that $R$ is the vector space with basis the set of length-0 paths, hence has dimension equal to the cardinality of $Q_0$, while $A$ is the vector space with basis the set of arrows of $Q$. If we define $e_ie_j=\delta_{i,j}e_i$, where $\delta_{i,j}$ is the \emph{Kronecker delta}, $R$ becomes a commutative $\ka$-algebra. Furthermore, if we define $e_i\alpha=\delta_{i,h(\alpha)}\alpha$ and $\alpha e_i=\delta_{i,t(\alpha)}\alpha$, then $A$, and actually every $A^\ell$ with $\ell>0$, becomes an $R$-$R$-bimodule.

\begin{definition} The \emph{path algebra} of $Q$ is the $\ka$-vector space $\usualRQ$ consisting of all finite linear combinations of paths in $Q$, that is,
\begin{equation}\label{eq:def-R<Q>}
\usualRQ=\underset{\ell=0}{\overset{\infty}{\bigoplus}}A^\ell.
\end{equation}
The \emph{complete path algebra} of $Q$ is the $\ka$-vector space $\completeRQ$ consisting of all possibly infinite linear combinations of paths in $Q$,
that is,
\begin{equation}\label{eq:def-R<<Q>>}
\completeRQ=\underset{\ell=0}{\overset{\infty}{\prod}}A^\ell.
\end{equation}
Both $\usualRQ$ and $\completeRQ$ have their multiplications induced by concatenation of paths (the product of two paths is their concatenation if they can be concatenated, and 0 if they cannot be concatenated). In terms of the homogeneous components in the decomposition \eqref{eq:def-R<<Q>>} (resp. \eqref{eq:def-R<Q>}), the multiplication of two elements of $\completeRQ$ (resp. $\usualRQ$) resembles the multiplication of formal power series (resp. polynomials). That is, if we have $u=\sum_{\ell\geq 0}u^{(\ell)}$ and $v=\sum_{\ell\geq0}v^{(\ell)}$, with $u^{(\ell)},v^{(\ell)}\in A^\ell$ for every $\ell\geq 0$, then
$$
uv=\sum_{\ell\geq 0}\sum_{\ell_1+\ell_2=\ell}u^{(\ell_1)}v^{(\ell_2)}.
$$
(In this equality, the right-hand side is a well-defined element of $\completeRQ$: for $\ell\geq 0$ fixed, we have $\sum_{\ell_1+\ell_2=\ell}u^{(\ell_1)}v^{(\ell_2)}=\sum_{k=0}^\ell u^{(k)}v^{(k-\ell)}$.)
\end{definition}

Note that $\usualRQ$ is a $\ka$-subalgebra of $\completeRQ$. Actually, $\usualRQ$ is dense in $\completeRQ$ under the \emph{$\maxid$-adic topology} of $\completeRQ$. The fundamental system of open neighborhoods of this topology around $0$ is given by the powers of the two-sided ideal $\maxid$ of $\completeRQ$ generated by the arrows of $Q$.

We are ready to describe the three operations this survey article is about: quiver mutations, mutations of quivers with potentials, and flips of surface triangulations.

\subsection{Quiver mutations}\label{subsec:quiver-mutations}

\begin{definition} Let $Q$ be a quiver.
An \emph{$\ell$-cycle} on $Q$ is a path $\alpha_1\alpha_2\ldots \alpha_\ell$, with  $\ell>0$, such that $h(\alpha_1)=t(\alpha_\ell)$. A quiver is \emph{2-acyclic} if it does not have 2-cycles.
\end{definition}

Central in the definition of cluster algebras is the notion of \emph{quiver mutation}. This is a combinatorial operation on 2-acyclic quivers that can be described as an elementary 3-step procedure as follows. Start with a 2-acyclic quiver $Q$ and a vertex $i$ of $Q$.
\begin{itemize}
\item[(Step 1)] Every time we have an arrow $\alpha:j\to i$ and an arrow $\beta:i\to k$ in $Q$, add an arrow $[\beta\alpha]:j\to k$;
\item[(Step 2)] replace each arrow $\gamma$ incident to $i$ with an arrow $\gamma^*$ going in the opposite direction;
\item[(Step 3)] delete 2-cycles one by one (2-cycles may have been created when applying Step 1).
\end{itemize}
The result is a 2-acyclic quiver $\mu_i(Q)$, called the \emph{mutation of $Q$ with respect to $i$}. See Figure \ref{Fig:example_quiver_mutation} for an example.
        \begin{figure}[!h]
                \centering
                \includegraphics[scale=.65]{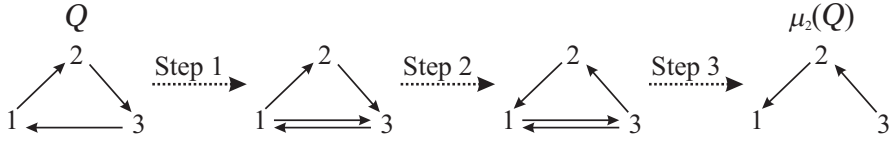}\caption{The three steps of quiver mutation. Here we are applying $\mu_2$.}\label{Fig:example_quiver_mutation}
        \end{figure}

\subsection{Mutations of quivers with potentials}\label{subsec:QP-mutations}

In a representation-theoretic approach to cluster algebras, Derksen-Weyman-Zelevinsky developed in \cite{DWZ1} a \emph{mutation theory of quivers with potentials}, which lifts quiver mutation from the combinatorial to the algebraic level.

\begin{definition} Let $Q$ be a quiver. An element $S$ of $\completeRQ$ is called a \emph{potential} if it is a possibly infinite linear combination of cycles of $Q$, with the property that no two different cycles appearing in $S$ with non-zero coefficient can be obtained from each other by rotation. If $S$ is a potential on $Q$, we say that the pair $(Q,S)$ is a \emph{quiver with potential}, or simply a \emph{QP}.
\end{definition}

Quiver mutation is lifted to the algebraic level of QPs by providing lifts of the three steps described in Subsection \ref{subsec:quiver-mutations}. Among the three steps, the one that turns out to be the hardest to lift is Step 3: one needs an algebraic procedure to delete 2-cycles algebraically. The procedure is provided by a technical result (Theorem \ref{thm:splittingthm} below) that requires some preparation.

\begin{definition}\label{def:QP-stuff} Let $Q$ and $Q'$ be quivers with the same vertex set $Q_0=Q_0'$.
\begin{enumerate}
\item Two potentials $S$ and $S'$ on $Q$ are \emph{cyclically-equivalent} if $S-S'$
lies in the closure of the vector subspace of $\completeRQ$ spanned by all the elements of the form $\alpha_1\ldots \alpha_\ell-\alpha_2\ldots \alpha_\ell\alpha_1$ with $\alpha_1\ldots
\alpha_\ell$ a cycle of positive length.
\item We say that two QPs $(Q,S)$ and $(Q',S')$ are \emph{right-equivalent} if there exists a \emph{right-equivalence} between them, that is, a $\ka$-algebra isomorphism $\varphi:\completeRQ\rightarrow \complete{Q'}$ satisfying $\varphi(e_i)=e_i$ for all $i\in Q_0=Q_0'$, and such that $\varphi(S)$ is cyclically-equivalent to $S'$.
\item For each arrow $\alpha\in Q_1$ and each cycle $\alpha_1\ldots \alpha_\ell$ in $Q$ we define the \emph{cyclic derivative}
$$
\partial_\alpha(\alpha_1\ldots \alpha_\ell)=\underset{k=1}{\overset{\ell}{\sum}}\delta_{\alpha,\alpha_k}\alpha_{k+1}\ldots \alpha_\ell\alpha_1\ldots \alpha_{k-1},
$$
and extend $\partial_\alpha$ by linearity and continuity so that $\partial_\alpha(S)$ is defined for every potential $S$.
\item The \emph{Jacobian ideal} $J(S)$ is the topological closure of the two-sided ideal of $\completeRQ$ generated by $\{\partial_\alpha(S)\suchthat \alpha\in Q_1\}$, and the
\emph{Jacobian algebra} $\jacobqs$ is the quotient algebra $\completeRQ/J(S)$.
\item A QP $(Q,S)$ is \emph{trivial} if $S\in A^2$ and $\{\partial_\alpha(S)\suchthat \alpha\in Q_1\}$ spans $A$ as a $\ka$-vector space.
\item A QP $(Q,S)$ is \emph{reduced} if the degree-$2$ component of $S$ is $0$, that is,
if the expression of $S$ involves no $2$-cycles.
\item The \emph{direct sum} $Q\oplus Q'$ is the quiver whose vertex set is $Q_0=Q_0'$ and whose arrow set is the disjoint union $Q_1\sqcup Q_1'$, with the tail and head functions defined in the obvious way.
\item The \emph{direct sum} of two QPs $(Q,S)$ and $(Q',S')$ is the QP $(Q,S)\oplus(Q',S')=(Q\oplus Q',S+S')$.
\end{enumerate}
\end{definition}

\begin{proposition}
\cite{DWZ1} If $\varphi:\completeRQ\rightarrow \complete{Q'}$ is a right-equivalence between $(Q,S)$ and $(Q',S')$, then $\varphi$ sends $J(S)$ onto $J(S')$ and therefore induces an algebra isomorphism $\jacobqs\rightarrow\mathcal{P}(Q',S')$.
\end{proposition}

\begin{theorem}
\cite{DWZ1}\label{thm:splittingthm} For every QP $(Q,S)$ there exist a trivial QP
$\qstriv{Q}{S}$ and a reduced QP $\qsred{Q}{S}$ such that $(Q,S)$ is right-equivalent to the direct sum $\qstriv{Q}{S}\oplus\qsred{Q}{S}$. 
The right-equivalence class of each of the QPs $\qstriv{Q}{S}$ and $\qsred{Q}{S}$ is determined by the right-equivalence class of $(Q,S)$.
\end{theorem}

\begin{proof}[On the proof] The proof of the second statement (uniqueness of $\qstriv{Q}{S}$ and $\qsred{Q}{S}$ up to right-equivalence) is rather long and relies on a series of technical (albeit elementary) preliminary results. One of the problems is that, for an arbitrary right-equivalence $\varphi:\complete{Q'\oplus C'}\rightarrow\complete{Q''\oplus C''}$ between $(Q',W')\oplus(C',T')$ and $(Q'',W'')\oplus(C'',T'')$, with $(Q',W'),(Q'',W'')$ reduced QPs and $(C',T'),(C'',T'')$ trivial QPs, it is not necessarily true that $\varphi$ restricts to an isomorphism $\complete{Q'}\rightarrow\complete{Q''}$, which means that the restriction of $\varphi$ to $\complete{Q'}$ is not necessarily a right-equivalence between $(Q',W')$ and $(Q'',W'')$.

Let us discuss Derksen-Weyman-Zelevinsky's proof of existence of $\qstriv{Q}{S}$ and $\qsred{Q}{S}$. Using basic linear algebra, it is easy to show that there exists an $R$-$R$-bimodule automorphism $f$ of $A$ such that the algebra automorphism $\psi$ of $\completeRQ$ induced by $f$ sends $S$ to a potential cyclically equivalent to
\begin{equation}
W=\sum_{k=1}^Na_Nb_N+W^{(\geq 3)}
\end{equation}
for some $N\geq 0$, some set $\{a_1,b_1,\ldots,a_N,b_N\}$ of $2N$ distinct arrows of $Q$ such that each $a_kb_k$ is a 2-cycle, and some potential $W^{(\geq 3)}\in\maxid^3$. Up to cyclical equivalence one can actually assume that
\begin{equation}\label{eq:pot-normal-form}
W=\left(\sum_{k=1}^Na_Nb_N+a_ku_k+b_kv_k\right)+W'
\end{equation}
for some $u_k$ and $v_k$ belonging to $\maxid^\ell$ for some $\ell\geq 2$, and some potential $W'\in\maxid^3$ that does not involve any of the arrows $a_1,b_1,\ldots,a_N,b_N$. If $u_k=v_k=0$ for all $k=1,\ldots,N$, then $(Q,W)$ is already reduced and the existence of $\qstriv{Q}{S}$ and $\qsred{Q}{S}$ is established. Otherwise, one uses \eqref{eq:pot-normal-form} to define an algebra automorphism $f$ of $\completeRQ$ by setting
\begin{equation}\nonumber
f(a_k)=a_k-v_k, \ \ \ f(b_k)=b_k-u_k, \ \ \ \text{for $k=1,\ldots,N$},
\end{equation}
and $f(c)=c$ for every arrow $c\in Q_1\setminus\{a_1,b_1,\ldots,a_N,b_N\}$. A little algebraic manipulation shows that $f(W)$ is cyclically equivalent to a potential of the form \eqref{eq:pot-normal-form}, but with the corresponding factors $u_k$ and $v_k$ belonging to a power $\maxid^{\ell'}$ with $\ell'>\ell$. This property and the fact that the \emph{depth} of $f$ is at least $\ell-1$ allow a recursive construction of a sequence $(f_n)_{n>0}$ of algebra automorphisms of $\complete{Q}$ such that the limit $\varphi=\lim_{n\to\infty}f_nf_{n-1}\ldots f_2f_1$ is a well-defined algebra automorphism of $\complete{Q}$ that sends $W$ to a potential cyclically-equivalent to a potential $\widetilde{W}$ of the form \eqref{eq:pot-normal-form} with all factors $u_k$ and $v_k$ equal to zero\footnote{This is the reason why we work with the complete path algebra $\complete{Q}$ rather than with the path algebra $\usual{Q}$: the only way to ensure convergence of this limit process for arbitrary $(Q,S)$ is to work with the complete path algebra.}. Setting $Q_{\operatorname{red}}$ (resp. $Q_{\operatorname{triv}}$) to be the subquiver of $Q$ whose arrow set is $Q_1\setminus\{a_1,b_1,\ldots,a_N,b_N\}$ (resp. $\{a_1,b_1,\ldots,a_N,b_N\}$), and $S_{\operatorname{red}}=\widetilde{W}-\sum_{k=1}^Na_kb_k$ (resp. $S_{\operatorname{triv}}=\sum_{k=1}^Na_kb_k$), we see that $\qsred{Q}{S}$ and $\qstriv{Q}{S}$ are a reduced and a trivial QP, respectively, and the composition $\varphi\psi$ is a right-equivalence between $(Q,S)$ and $\qsred{Q}{S}\oplus\qstriv{Q}{S}$.
\end{proof}

\begin{definition} In the situation of Theorem \ref{thm:splittingthm}, the QPs $\qsred{Q}{S}$ and $\qstriv{Q}{S}$ are called, respectively, the \emph{reduced part} and the \emph{trivial
part} of $(Q,S)$.
\end{definition}

We now turn to the definition of mutation of a QP. Let $(Q,S)$ be a QP on the vertex set $Q_0$ and let $i\in Q_0$. Assume that $Q$ is 2-acyclic. Let $\widetilde{\mu}_i(Q)$ be the quiver obtained right after applying the first two steps of quiver mutation, but before applying Step 3. Replacing $S$ if necessary with a cyclically equivalent potential, we assume that every cycle appearing in the expression of $S$ starts at a vertex different from $i$. Then we define $[S]$ to be the potential on $\widetilde{\mu}_i(Q)$ obtained from
$S$ by replacing each length-2 path $\beta\alpha$ of $Q$ such that $h(\alpha)=i=t(\beta)$, with the arrow $[\beta\alpha]$ of $\widetilde{\mu}_i(Q)$. Also, we define $\Delta_i(Q)=\sum \alpha^*\beta^*[\beta\alpha]$,
where the sum runs over all length-2 paths $\beta\alpha$ of $Q$ such that $h(\alpha)=i=t(\beta)$. Finally, we set
$\widetilde{\mu}_i(S)=[S]+\Delta_i(Q)$, which clearly is a potential on $\widetilde{\mu}_i(Q)$.

\begin{definition}\cite{DWZ1}\label{def:QP-mutation} Under the assumptions and notation just stated, we define the 
\emph{mutation} $\mu_i(Q,S)$ of $(Q,S)$ with respect to $i$ to be the reduced part of the QP
$\widetilde{\mu}_i(Q,S)=(\widetilde{\mu}_i(Q),\widetilde{\mu}_i(S))$.
\end{definition}

``Unfortunately", given a QP $(Q,S)$ with $Q$ 2-acyclic, the underlying quiver of the mutated QP $\mu_i(Q,S)$ is not necessarily 2-acyclic, its 2-acyclicity depends heavily on the potential $S$.

\begin{definition}\cite{DWZ1}
A QP $(Q,S)$ is \emph{non-degenerate} if $Q$ is 2-acyclic and the underlying quiver of the QP obtained after any possible sequence of QP-mutations is 2-acyclic.
\end{definition}

\begin{theorem}\cite{DWZ1}\label{thm:propertiesofQP-mutations} The following hold if $\ka$ is the ground field:\begin{enumerate}\item Mutations of QPs are well defined up to right-equivalence.
\item Mutations of QPs are involutive up to right-equivalence.
\item Every 2-acyclic quiver $Q$ admits a non-degenerate potential on it.
\item Finite-dimensionality of Jacobian algebras is invariant under QP-mutations.
\end{enumerate}
\end{theorem}

\begin{remark} Theorem \ref{thm:propertiesofQP-mutations} still holds for ground fields different from $\ka$. Indeed, parts (1), (2) and (4) hold over any ground field, while part (3) holds over any uncountable ground field.
\end{remark}

\subsection{Flips of triangulations}\label{subsec:triangulations}

\begin{definition}
A \emph{surface with marked points}, or simply a \emph{surface}, is a pair $\surf$, where $\Sigma$ is a compact connected oriented Riemann surface with (possibly empty) boundary, and $\marked$ is a non-empty finite subset of $\Sigma$ containing at least one point from each connected component of the boundary of $\Sigma$. We refer to the elements of $\marked$ as \emph{marked points}. The marked points that lie in the interior of $\Sigma$ are called \emph{punctures}, and the set of punctures of $\surf$ is denoted $\punct$.
\end{definition}

We think of $\marked$ as a prescribed set of vertices for triangulations of $\Sigma$. More formally:

\begin{definition} Let $\surf$ be a surface with marked points.
\begin{enumerate}\item An \emph{arc} on $\surf$, is a curve $i$ on $\Sigma$ such that:
\begin{itemize}
\item the endpoints of $i$ belong to $\marked$;
\item $i$ does not intersect itself, except that its endpoints may coincide;
\item the points in $i$ that are not endpoints do not belong to $\marked$ nor to the boundary of $\Sigma$;
\item $i$ does not cut out an unpunctured monogon nor an unpunctured digon.
\end{itemize}
\item Two arcs $i_1$ and $i_2$ are \emph{isotopic relative to $\marked$} if there exists a continuous function $H:[0,1]\times\Sigma\rightarrow\Sigma$ such that
    \begin{itemize}
    \item[(a)] $H(0,x)=x$ for all $x\in\Sigma$;
    \item[(b)] $H(1,i_1)=i_2$;
    \item[(c)] $H(t,m)=m$ for all $t\in I$ and all $m\in\marked$;
    \item[(d)] for every $t\in I$, the function $H_t:\Sigma\to\Sigma$ given by $x\mapsto H(t,x)$ is a homeomorphism.
    \end{itemize}
    Arcs will be considered up to isotopy relative to $\marked$, parametrization, and orientation.
\item Two arcs are \emph{compatible} if there are arcs in their respective isotopy classes that, except possibly for their endpoints, do not intersect.
\item An \emph{ideal triangulation} of $\surf$ is any maximal collection $\tau$ of pairwise compatible arcs.
\end{enumerate}
\end{definition}

\begin{remark} The adjective \emph{ideal} in the term \emph{ideal triangulation} comes from the connection with Teichm\"{u}ller theory, see \cite{FT}.
\end{remark}

If $\tau$ is any ideal triangulation, then, replacing if necessary each arc in $\tau$ with an isotopic one (relative to $\marked$), one can assume that any two arcs in $\tau$ intersect at most at their endpoints. This fact is less trivial than it may seem at first glance, see for example \cite{FHS}.

Any ideal triangulation $\tau$ of $\surf$ splits $\Sigma$ into \emph{triangles}. If $\surf$ has punctures, some of the triangles of $\tau$ may be \emph{self-folded}, see Figure \ref{Fig:triang_with_sf_triangle}. A self-folded triangle always contains a \emph{folded side}.
        \begin{figure}[!h]
                \centering
                \includegraphics[scale=.4]{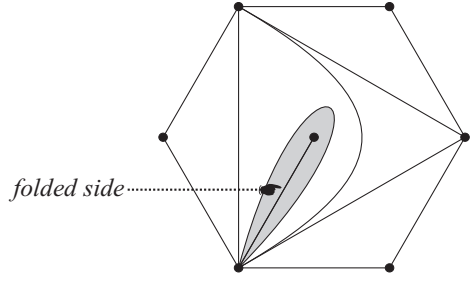}\caption{Ideal triangulation with a self-folded triangle (the self-folded triangle has been highlighted).
                }\label{Fig:triang_with_sf_triangle}
        \end{figure}
If $i\in\tau$ is an arc which is not the folded side of a self-folded triangle, then there exists a unique arc $j\neq i$ on $\surf$, such that $\sigma=(\tau\setminus\{i\})\cup \{j\}$ is an ideal triangulation of $\surf$. We say that $\sigma$ is obtained from $\tau$ by the \emph{flip} of the arc $i$. See Figure \ref{Fig:triangs_rel_by_flip}. Intuitively speaking, flipping an arc of a triangulation corresponds to the operation that replaces a diagonal of a square with the other diagonal.
        \begin{figure}[!h]
                \centering
                \includegraphics[scale=.4]{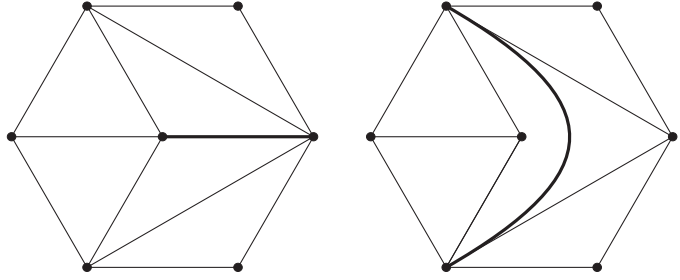}\caption{Two ideal triangulations related by a flip, the arcs involved in the flip have been drawn bolder than the rest.
                }\label{Fig:triangs_rel_by_flip}
        \end{figure}

Note that with the notion of triangulation we have thus far (that of ideal triangulation), it is not possible to flip folded sides of self-folded triangles. In order to be able to flip these, Fomin-Shapiro-Thurston introduced in \cite{FST} the concept of \emph{tagged triangulation}, a notion of triangulation which is more general than the notion of ideal triangulation we have defined above. The combinatorics of flips of tagged triangulations becomes rather subtle at some points, but the following does hold:

\begin{theorem}\label{thm:flips-of-tagged-triangs} Let $\surf$ be a surface with marked points.
\begin{enumerate}
\item\cite{FST} If $\tau$ is a tagged triangulation of $\surf$, and $i$ is a tagged arc belonging to $\tau$, then there exists a unique tagged arc $j\neq i$ on $\surf$ such that $\sigma=(\tau\setminus\{i\})\cup\{j\}$ is a tagged triangulation of $\surf$. In other words, every tagged arc in a tagged triangulation can be flipped.
\item\cite{Mosher} Every two ideal triangulations can be connected by a sequence of flips of ideal triangulations.
\item\cite{FST} If $\surf$ is not a closed surface with exactly one puncture, then every two tagged triangulations of $\surf$ can be connected by a sequence of flips.
\item\cite{FST} Every quiver $\qtau$ with $\tau$ a tagged triangulation is isomorphic to a quiver of the form $\qsigma$ with $\sigma$ an ideal triangulation.
\end{enumerate}
\end{theorem}

\section{The quiver with potential of a triangulation}\label{sec:QP-of-triangulations}

\subsection{The quiver of a triangulation}\label{subsec:quiver-of-triangulations}

Every ideal triangulation $\tau$ has a quiver $\qtau$ associated in a natural way. This was first observed by Fock-Goncharov \cite{FG}, Fomin-Shapiro-Thurston \cite{FST} and Gekhtman-Shapiro-Vainshtein \cite{GSV}. Let us describe $\qtau$ under the assumption that
\begin{equation}\label{eq:nice-triangulation}\text{every puncture of $\surf$ is incident to at least three arcs of $\tau$.}
\end{equation}
The vertices of $\qtau$ are the arcs of $\tau$, the arrows are drawn in the clockwise direction within each triangle of $\tau$. See Figure \ref{Fig:triangs_rel_by_flip_QP}.
        \begin{figure}[!h]
                \centering
                \includegraphics[scale=.55]{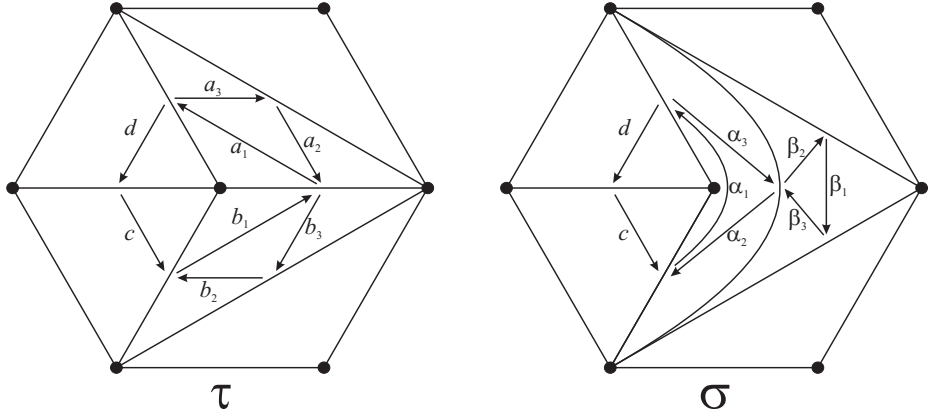}\caption{Two ideal triangulations related by a flip, with their associated quivers drawn on the surface.
                }\label{Fig:triangs_rel_by_flip_QP}
        \end{figure}

If an ideal triangulation $\tau$ is such that
\begin{equation}\label{eq:less-than-3-arcs}
\text{there are punctures incident to less than three arcs of $\tau$,}
\end{equation}
or more generally, if $\tau$ is a tagged triangulation, then
the definition of $\qtau$ is slightly more involved, but we stress that all triangulations, including the tagged ones, have naturally associated quivers. Let us also remark that the definition of the quivers of tagged triangulations is due to Fomin-Shapiro-Thurston alone.

\begin{theorem}\cite{FG, FST, GSV}\label{thm:flip->quiver-mutation} Let $\tau$ and $\sigma$ be ideal triangulations of $\surf$. If $\sigma$ is obtained from $\tau$ by the flip of an arc $i$, then $Q(\sigma)=\mu_i(Q(\tau))$. That is, if two ideal triangulations are related by a flip, then their associated quivers are related by the corresponding quiver mutation.
\end{theorem}

Thus, for example, the two quivers drawn in Figure \ref{Fig:triangs_rel_by_flip_QP} are related by quiver mutation. Fomin-Shapiro-Thurston have shown that Theorem \ref{thm:flip->quiver-mutation} is valid in the more general setting of tagged triangulations.

\subsection{The potential of a triangulation}\label{subsec:potential-of-triangulations}

We know that every triangulation has a quiver associated to it, and we know that flips of triangulations are compatible with mutations of quivers, in the sense that if two triangulations are related by a flip, then their associated quivers are related by the corresponding quiver mutation. Could this story be lifted to the level of QPs? To try and answer this question, we first notice that for each ideal triangulation $\tau$ of $\surf$ satisfying \eqref{eq:nice-triangulation}, the quiver $\qtau$ possesses two obvious types of cycles:
\begin{itemize}
\item 3-cycles arising from triangles $\triangle$ of $\tau$, and
\item simple cycles (that is, without repeated arrows) surrounding the punctures.
\end{itemize}
To avoid redundancies, in the following definition we consider cycles up to cyclical, that is, we take only one cycle per cyclical equivalence class of cycles.

\begin{definition}\cite{L1}\label{def:potential-for-tau} Let $\tau$ be a triangulation of $\surf$ satisfying \eqref{eq:nice-triangulation}. The potential $\stau$ associated to $\tau$ is the potential on $\qtau$ that results from adding all the 3-cycles that arise from triangles of $\tau$ and all the simple cycles that surround the punctures of $\surf$.
\end{definition}

\begin{remark}
\begin{enumerate}\item In the case when $\surf$ is a surface without punctures and non-empty boundary, the potentials $\stau$ were found independently by Assem-Br\"{u}stle-Charbonneau-Plamondon in \cite{ABCP}.
\item In \cite{L1}, the definition of $\stau$ was given for every ideal triangulation $\tau$, including those satisfying \eqref{eq:less-than-3-arcs}.
\end{enumerate}
\end{remark}

\begin{example} The potentials associated to the ideal triangulations $\tau$ and $\sigma$ shown in Figure \ref{Fig:triangs_rel_by_flip_QP} are $\stau=a_1a_2a_3+b_1b_2b_3+a_1b_1cd$ and $\ssigma=\alpha_1\alpha_2\alpha_3+\beta_1\beta_2\beta_3+\alpha_1cd$. The QPs $\qstau$ and $(\qsigma,\ssigma)$ turn out to be related by QP-mutation. This can be checked directly, or seen as a consequence of the following theorem.
\end{example}

\begin{theorem}\cite{L1}\label{thm:ideal-flip->QP-mut} Let $\tau$ and $\sigma$ be ideal triangulations of a surface with marked points $\surf$. \begin{enumerate}\item If $\tau$ and $\sigma$ are related by the flip of an arc $i$, then the QPs $\mu_i(\qtau,\stau)$ and $(\qtau,\stau)$ are right-equivalent. In other words, $(\qtau,\stau)$ and $(\qsigma,\ssigma)$ are related by the QP-mutation $\mu_i$.
\item If the boundary of $\Sigma$ is not empty, then all QPs $(\qtau,\stau)$ associated to the ideal triangulations of $\surf$ are non-degenerate.
\end{enumerate}
\end{theorem}

\begin{proof}[On the proof] 
Note that if $\tau$ and $\sigma$ are ideal triangulations related by the flip of an arc $i\in\tau$, then $i$ cannot be the folded side of a self-folded triangle of $\tau$.

Fomin-Shapiro-Thurston introduce three different kinds of \emph{puzzle pieces} --planar unpunctured triangles, planar punctured digons triangulated in a very specific way, and planar twice-punctured monogons also triangulated in a specific way. These puzzle pieces come in handy because any given ideal triangulation of an arbitrary surface $\surf$ can be obtained by gluing some set of puzzle pieces.

Given a \emph{puzzle-piece decomposition} of an ideal triangulation $\tau$, any flip of an arc of $\tau$ occurs either inside a puzzle piece, or at the arc shared by two puzzle pieces that are glued together. These possibilities for a flip comprise a basic list of cases to be considered in the proof of Part (1) of Theorem \ref{thm:ideal-flip->QP-mut}, although a couple of slight subtleties have to be considered. The first subtlety concerns the fact that a puzzle-piece decomposition gives not only a gluing of pairs of sides of puzzle pieces, but in many cases also an identification of different vertices of puzzle pieces as the same marked point in $\surf$. Such identification affects how the potential $\stau$ looks like locally, in that the local configuration of $\stau$ around a pair of glued puzzle pieces is not determined only by the gluing of the side(s) shared by the puzzle pieces, but depends also on how the vertices of these are identified to obtain $\tau$. The second subtlety concerns the fact that some gluing of puzzle pieces yield some 2-cycles that ``are not seen" by the quiver $\qtau$.

The two mentioned subtleties make it necessary to refine the basic list of cases to be checked for the proof of Part (1) of Theorem \ref{thm:ideal-flip->QP-mut}. After such refinement, one has a finite (albeit larger) list of cases. Note that, by what we said in the first paragraph, none of these cases involves the flip of a folded side. In the cases that are indeed considered, one reads the QPs $(\qtau,\stau)$ and $(\qsigma,\ssigma)$ on the one hand, and the mutation $\mu_i(\qtau,\stau)$ on the other. Then an explicit right-equivalence between $\mu_i(\qtau,\stau)$ and $(\qtau,\stau)$ is given.

The proof of part (2) is done by induction on the number of punctures of $\surf$, using the non-emptiness of the boundary in an essential way. Given an unpunctured surface $(\Sigma,\marked_0)$ with non-empty boundary, it is straightforward to see that any ideal triangulation $\tau_0$ of $(\Sigma,\marked_0)$ has the property that its QP $(Q(\tau_0),S(\tau_0))$ is \emph{rigid}. Then we add punctures to $(\Sigma,\marked_0)$ one by one, thus obtaining an $n$-punctured surface $(\Sigma,\marked_n)=(\Sigma,\marked_0\cup\{p_1,\ldots,p_n\})$ for each $n\geq 1$. Every time we add a puncture we also complete the ideal triangulation $\tau_{n-1}$ of $(\Sigma,\marked_{n-1})$ to an ideal triangulation $\tau_n$ of $(\Sigma,\marked_n)$ in a very specific. We then show the rigidity of $(Q(\tau_n),S(\tau_n))$ under the assumption that $(Q(\tau_{n-1}),S(\tau_{n-1}))$ is rigid. By Part (1) of Theorem \ref{thm:ideal-flip->QP-mut}, this establishes the rigidity of the QP associated to any ideal triangulation of $(\Sigma,\marked_n)$, since rigidity is preserved by QP-mutations thanks to a result of Derksen-Weyman-Zelevinsky. And since another result of Derksen-Weyman-Zelevinsky shows that rigid QPs are non-degenerate, Part (2) of Theorem \ref{thm:ideal-flip->QP-mut} follows.
\end{proof}

Potentials corresponding to arbitrary tagged triangulations were not even mentioned in \cite{L1}, despite the fact that it is possible to read an `obvious' potential from any tagged triangulation. The reason is that, at that moment, a ``tagged version" of Part (1) of Theorem \ref{thm:ideal-flip->QP-mut} had not been established: it was not clear whether the `obvious' potentials would have the property that arbitrary tagged triangulations related by a flip would have QPs related by the corresponding QP-mutation. Among other things, a ``tagged version" of Part (1) of Theorem \ref{thm:ideal-flip->QP-mut} would allow us to deduce a fact that is not an immediate consequence of Theorem \ref{thm:ideal-flip->QP-mut}, namely, that the QPs $(\qtau,\stau)$ associated to triangulations of surfaces with empty boundary are non-degenerate.

Potentials for arbitrary tagged triangulations were defined in \cite{CI-LF} under the assumption that the underlying surface $\Sigma$ has non-empty boundary; but even with this assumption, the corresponding ``tagged version" of Theorem \ref{thm:ideal-flip->QP-mut} was not proved for all flips of tagged triangulations, but only for some of them. Indeed, it was proved in \cite{CI-LF} that for every two tagged triangulations $\tau$ and $\sigma$ of a surface with non-empty boundary, there exists a sequence $(\tau=\tau_0,\tau_1,\tau_2,\ldots,\tau_{n-1},\tau_n=\sigma)$ of tagged triangulations with the property that each $\tau_\ell$ is obtained from $\tau_{\ell-1}$ by the flip of a tagged arc $i_{\ell}\in\tau_{\ell-1}$ and the QP $\mu_{i_{\ell}}(\qtau,\stau)$ is right-equivalent to the QP $(\qtau,\stau)$. However, for two arbitrary tagged triangulations $\tau$ and $\sigma$ related by the flip of a tagged arc $i$, it was not proved that the sequence of flips just described can always be taken to be the sequence $(\tau=\tau_0,\tau_1=\sigma)$.

The referred assumption on the boundary was removed in \cite{L4}, where potentials were defined for all tagged triangulations of surfaces, including both the surfaces with boundary and the surfaces without boundary, and the ``tagged version" of Theorem \ref{thm:ideal-flip->QP-mut} was proved for all flips of tagged triangulations:

\begin{theorem}\label{thm:tagged-flip<->mutation} Let $\surf$ be a surface with marked points. Suppose $\surf$ is not a sphere with less than five punctures. If $\tau$ and $\sigma$ are tagged triangulations of $\surf$ related by the flip of a tagged arc $i$, then
\begin{enumerate}
\item\cite{L4} the QPs $\mu_i(\qtau,\stau)$ and $(\qsigma,\ssigma)$ are right-equivalent if $\surf$ is not a sphere with exactly five punctures;
\item\cite{G-LF-S} the QP $\mu_i(\qtau,\stau)$ is right-equivalent to $(\qsigma,\lambda\ssigma)$ for some non-zero scalar $\lambda$ if $\surf$ is a sphere with exactly five punctures.
\end{enumerate}
Consequently, all QPs $(\qtau,\stau)$ associated to the tagged triangulations of $\surf$ are non-degenerate.
\end{theorem}

\begin{proof}[On the proof] The proofs of (1) and (2) differ at a crucial point (note that 
(2) is weaker than (1)).
We only sketch the proof of (1). Every tagged triangulation $\tau$ gives rise to a function $\epsilon_\tau:\punct\rightarrow\{-1,1\}$, called \emph{weak signature of $\tau$}, that takes the value $-1$ at a puncture $p$ if and only if at least two notches of tagged arcs in $\tau$ are incident to $p$.

It is easy to deduce from Theorem \ref{thm:ideal-flip->QP-mut} that if $\tau$ and $\sigma$ are tagged triangulations whose weak signatures $\epsilon_\tau$ and $\epsilon_\sigma$ are equal, and $\tau$ and $\sigma$ are related by the flip of a tagged arc $i$, then the QPs $\mu_i(\qtau,\stau)$ and $(\qsigma,\ssigma)$ are right-equivalent. This reduces the proof of Theorem \ref{thm:tagged-flip<->mutation} to the case of tagged triangulations $\tau$ and $\sigma$ related by a flip, but with different weak signatures $\epsilon_\tau$ and $\epsilon_\sigma$. A moment of reflection shows that in this case $\epsilon_\tau$ and $\epsilon_\sigma$ differ at exactly one puncture, say $p$.

One can assume, without loss of generality, that $\varepsilon_{\tau}(p)=1=-\varepsilon_\sigma(p)$ (this is because every flip is an involution and every QP-mutation is an involution up to right-equivalence). Applying a combinatorial procedure of ``deletion of notches", one can further assume that $\tau$ is an ideal triangulation. The proof of Theorem \ref{thm:tagged-flip<->mutation} is hence reduced to showing that if $\sigma$ is the tagged triangulation obtained by flipping the folded side $i$ of an arbitrary self-folded triangle of an ideal triangulation $\tau$, then $\mu_i(\qtau,\stau)$ is right-equivalent to $(\qsigma,\ssigma)$. The proof of this last implication is rather involved and relies crucially on \cite[Theorem 6.1]{L4}, a technical result that guarantees the existence of a right-equivalence between $(\qtau,\stau)$ and $(\qtau,W(\tau))$ for certain specific potential $W(\tau)$ that, on the other hand, has the property that $(\qtau,W(\tau))$ can be easily seen to be right-equivalent to $\mu_i(\qsigma,\ssigma)$.
\end{proof}

\begin{remark} The crucial right-equivalence in the proof of Theorem \ref{thm:tagged-flip<->mutation}, namely, the right-equivalence between $(\qtau,\stau)$ and $(\qtau,W(\tau))$ above, is not exhibited explicitly: for one triangulation it is constructed as the limit of certain algebra automorphisms of $\complete{\qtau}$, and for the other triangulations it is only shown to exist. This is very unlike the proof of Part (1) of Theorem \ref{thm:ideal-flip->QP-mut}, where, despite the division into cases, all right-equivalences are defined explicitly.
\end{remark}

\begin{remark}\label{rem:sphere-with-4-puncts}
The sphere with four punctures has been dealt with in \cite{GG} and \cite{G-LF-S}. We stress the fact that, in order to obtain a non-degenerate potential under Definition \ref{def:potential-for-tau}, it is strictly necessary to multiply exactly one of the cycles around punctures by a scalar $\lambda\in\ka\setminus\{0,1\}$ (the rest of the cycles are still taken as are, that is, multiplied by 1).
\end{remark}

%

\section{Dimension and representation type of Jacobian algebras}\label{sec:Jacobian-algebras}

From the perspective of representation theory of associative algebras, there are several natural questions one can ask regarding the Jacobian algebras of the QPs $\qstau$. For example: Are they finite-dimensional?, are they tame/wild? 

\begin{theorem}\cite{L1}\label{thm:non-empty-bound-fin-dim} Let $\surf$ be a surface with non-empty boundary (and any number of punctures). Then for any ideal triangulation $\tau$ of $\surf$, the Jacobian algebra $\jacob{\qtau}{\stau}$ has finite dimension over $\ka$.
\end{theorem}

\begin{theorem}\cite{Ladkani}\label{thm:empty-bound-fin-dim} Let $\surf$ be a surface with empty boundary. Suppose that $\surf$ is not a sphere with less than 5 punctures. Then for any ideal triangulation $\tau$ of $\surf$, the Jacobian algebra $\jacob{\qtau}{\stau}$ has finite dimension over $\ka$.
\end{theorem}

\begin{remark}\begin{enumerate}
\item Theorem \ref{thm:non-empty-bound-fin-dim} can be either proved independently of Theorem \ref{thm:empty-bound-fin-dim}, or deduced from it via \emph{restriction of QPs}.
\item In the case of unpunctured surfaces with non-empty boundary, Theorem \ref{thm:non-empty-bound-fin-dim} was proved by Assem-Br\"{u}stle-Charbonneau-Plamondon in \cite{ABCP} independently of \cite{L1} and \cite{Ladkani}.
\item For polygons with at most one puncture, the finite-dimensionality of the Jacobian algebras $\jacob{\qtau}{\stau}$ was already known to Caldero-Chapoton-Schiffler \cite{CCS} and Schiffler \cite{Schiffler}, although in the referred papers the authors did not work with complete path algebras or potentials.
\item Theorem \ref{thm:empty-bound-fin-dim} is due to Ladkani. In the case of spheres with at least five punctures, it was proved independently by Trepode--Valdivieso-D\'{i}az in \cite{TV}.
\item For the sphere with exactly four punctures, the finite-dimensionality of Jacobian algebras of non-degenerate potentials follows from  \cite{GG} (where these algebras are shown to be tubular) and \cite{G-LF-S}. See Remark \ref{rem:sphere-with-4-puncts} above.
\end{enumerate}
\end{remark}

Let us turn to the problem of whether the Jacobian algebras $\jacob{\qtau}{\stau}$ are tame or wild.

\begin{definition} Let $\Lambda$ be a finite-dimensional associative $\mathbb{C}$-algebra.
\begin{enumerate}
\item We say that $\Lambda$ is \emph{tame} if for each
dimension vector $\mathbf{d}$ there are finitely many
$\Lambda$-$\ka[X]$-bimodules $M_1,\ldots,M_t$,
free of finite rank as right $\ka[X]$-modules, such that
every indecomposable $\Lambda$-module $N$ with $\operatorname{dim}(N) = \mathbf{d}$
is isomorphic to a $\Lambda$-module of the form
$$
M_i \otimes_{\ka [X]} \left(\ka [X]/(X-\lambda)\right)
$$
for some $1 \leq i \leq t$ and some $\lambda \in \ka $.
\item We say that $\Lambda$ is \emph{wild} if there is a
$\Lambda$-$\ka\langle X,Y\rangle$-bimodule $M$, free of finite rank as
right $\ka\langle X,Y\rangle$-module, such that
the exact functor
$$
M \otimes_{\ka\langle X,Y\rangle} - \colon \operatorname{mod}(\ka\langle X,Y\rangle) \to \operatorname{mod}(\Lambda)
$$
sends indecomposable modules to indecomposable ones and pairwise non-isomorphic modules to pairwise non-isomorphic ones. Here, $\ka\langle X,Y\rangle$ denotes the free $\ka$-algebra in two (non-commuting) generators $X$ and $Y$.
\end{enumerate}
\end{definition}

A famous result of Drozd asserts that any given finite-dimensional $\mathbb{C}$-algebra is either tame or wild, and not simultaneously tame and wild. This is Drozd's \emph{tame/wild dichotomy theorem}.

In the case when $\surf$ is a surface with non-empty boundary and without punctures, Assem-Br\"{u}stle-Charbonneau-Plamondon have shown in \cite{ABCP} that $\jacob{\qtau}{\stau}$ is a \emph{gentle algebra}, and this implies its tameness, for gentle algebras are well-known to be tame. More generally, we have:

\begin{theorem}\cite{G-LF-S}\label{thm:tame-jacobian-algebras}\begin{enumerate}\item For any QP $(Q,S)$, if  $\jacob{Q}{S}$ is tame, then $\jacob{\mu_i(Q}{S)}$ is tame as well.
\item Any surface with marked points has a triangulation $\tau$ such that the Jacobian algebra $\jacob{\qtau}{\stau}$ is tame.
\item Consequently, for every surface $\surf$ and every tagged triangulation $\tau$ of $\surf$, the Jacobian algebra $\jacob{\qtau}{\stau}$ is tame.
\end{enumerate}
\end{theorem}

\begin{proof}[On the proof] The proof of Part (1) of Theorem \ref{thm:tame-jacobian-algebras} presents a challenge: Quotients of finite-rank free $\ka[X]$-modules are not always free $\ka[X]$-modules, not every inclusion of $\ka[X]$-modules is a section, and not every surjective morphism of $\ka[X]$-modules is a retraction. So, one cannot define the mutation of $\jacob{Q}{S}$-$\ka[X]$-bimodules in the ``obvious" way, since the mutation process involves taking certain cokernel, as well as a section and a retraction. What Geiss-LF-Schr\"oer end up doing in \cite{G-LF-S} is the following: For each localization $R$ of $\ka[X]$, they define the mutation of a $\jacob{Q}{S}$-$R$-bimodule $M$ (assumed to be a finite-rank free right $R$-module), as a $\jacob{\mu_i(Q}{S)}$-$R'$-bimodule $M'$ for some other localization $R'$ of $\ka[X]$ that depends on $M$ (as a right $R'$-module, $M'$ turns out to be finite-rank free). With such definition, to show Part (1) of Theorem \ref{thm:tame-jacobian-algebras} they then use a characterization of tameness given by Dowbor-Skowro\'nski \cite{DS}  in terms of localizations of the polynomial ring $\ka[X]$.

If $\surf$ is a surface with non-empty boundary (and any number of punctures), different from a (punctured) monogon, then there exists a triangulation $\tau$ of $\surf$ such that the Jacobian algebra $\jacob{\qtau}{\stau}$ is a \emph{clannish algebra}. In the case of (punctured) monogons, one can always find a triangulation such that the Jacobian algebra is a \emph{deformation} of a \emph{skewed-gentle algebra}. If, on the other hand, the boundary of $\surf$ is empty, then there exists a triangulation $\tau$ of $\surf$ such that the Jacobian algebra $\jacob{\qtau}{\stau}$ is a deformation of a gentle algebra. In each of the three situations just described, it is easy to find an explicit triangulation $\tau$ with the stated property.

Part (3) is an immediate consequence of Parts (1) and (2).
\end{proof}

Theorem \ref{thm:tame-jacobian-algebras} tells us that for every QP of the form $(\qtau,\stau)$ the associated Jacobian algebra is tame. Something stronger is true: except for a couple of surfaces, for any non-degenerate potential on the quiver of a triangulation the Jacobian algebra is tame. To be precise:

\begin{theorem}\cite{G-LF-S}\label{thm:always-tame} Let $\surf$ be a surface with marked points and let $\tau$ be a triangulation of $\surf$.
\begin{enumerate}\item If $\surf$ is not a torus with exactly one marked point, then for any non-degenerate potential $S$ on $\qtau$, the Jacobian algebra $\jacob{\qtau}{S}$ is tame.
 \item If $\surf$ is a torus with exactly one marked point (hence $\Sigma$ has either empty boundary or exactly one boundary component), then $\qtau$ admits a non-degenerate potential $W$ such that $\jacob{\qtau}{W}$ is a wild algebra.
 \end{enumerate}
\end{theorem}

\begin{proof}[On the proof] To prove (1), in \cite{G-LF-S} we show that if $\surf$ is not a torus with exactly one marked point, nor a sphere with less than 6 punctures, nor an annulus with exactly two marked points, then $\surf$ admits a triangulation $\sigma$ such that every puncture is incident to at least $4$ arcs of $\sigma$, and moreover, the quiver $\qsigma$ does not have double arrows. These properties of $\sigma$ imply that
\begin{itemize}
\item any non-degenerate potential on $\qsigma$ is right-equivalent to $S'+S''$ for some potential $S''$ involving only cycles of length at least 4, where $S'$ is the sum of all 3-cycles of $\qsigma$;
\item the Jacobian algebra $\jacob{\qsigma}{S'}$ is gentle (albeit possibly infinite-dimensional).
\item (each \emph{truncation} of) the Jacobian algebra $\jacob{\qsigma}{S'+S''}$ is a deformation of (the corresponding truncation) of $\jacob{\qsigma}{S'}$.
\end{itemize}
Using a theorem of Crawley-Boevey \cite{CB}, this allows to deduce the tameness of $\jacob{\qsigma}{S'+S''}$. Part (1) of Theorem \ref{thm:always-tame}, combined with Parts (2), (3) and (4) of Theorem \ref{thm:flips-of-tagged-triangs}, then implies that for the arbitrarily given triangulation $\tau$ one has that for any non-degenerate potential $S$ on $\qtau$, the algebra $\jacob{\qtau}{S}$ is tame.

The spheres with less than 6 punctures and the annulus with exactly two marked points are treated separately.

For Part (2), consider the quivers
$$
T_1=
\vcenter{\xymatrix@-1.2pc{
&&2 \ar@<0.4ex>[dddrr]^{\gamma_1}\ar@<-0.4ex>[dddrr]_{\gamma_2}\\
&&&&\\
&&&&\\
1 \ar@<0.4ex>[rruuu]^{\alpha_1}\ar@<-0.4ex>[rruuu]_{\alpha_2} &&&&3 \ar@<0.4ex>[llll]^{\beta_1}\ar@<-0.4ex>[llll]_{\beta_2}
}}
\ \ \ \ \  \text{and} \ \ \ \ \
T_2 =
\vcenter{\xymatrix{
         &&1\ar@<-.8ex>[dd]_{\alpha_1}\ar@<.8ex>[dd]^{\alpha_2}&\\
3\ar[rru]^(.6){\beta_1}\ar@<1.8ex>@/^{4pc}/[rrrr]^{\delta}&&      &&4\ar[llu]_(.6){\beta_2} \\
         &&2\ar[llu]^{\gamma_1}\ar[rru]_{\gamma_2}
}}
$$
It is straightforward to see that if $\surf$ is a torus with exactly one marked point and $\tau$ is any triangulation of $\surf$, then $\qtau$ is isomorphic to $T_1$ if $\partial\Sigma=\varnothing$, and to $T_2$ if $\partial\Sigma\neq\varnothing$. It is also easy to see that for any vertex $i$ of $T_\ell$ ($\ell=1,2$), the quiver $\mu_i(T_\ell)$ is isomorphic to $T_\ell$ via a quiver isomorphism $\pi_i:T_\ell\rightarrow \mu_i(T_\ell)$ that acts as a uniquely determined permutation on the (common) vertex set (the quiver isomorphism is unique for $\ell=2$, and there is a little choice involved at the arrow level for $\ell=1$). Let
\begin{equation}\label{eq:wild-pot-for-torus-without-bound}
W_1=\alpha_1\beta_1\gamma_2+\alpha_1\beta_2\gamma_1+\alpha_2\beta_1\gamma_1\in\complete{T_1}
\end{equation}
$$
\text{and}
$$
\begin{equation}\label{eq:wild-pot-for-torus-w-bound}
W_2=\alpha_1\beta_1\gamma_1+\alpha_1\beta_2\gamma_2+\alpha_2\beta_2\delta\gamma_1\in\complete{T_2}.
\end{equation}
Direct computation shows that for any vertex $i$ of $T_\ell$, the image of $W_\ell$ under the isomorphism $\pi_i$ is precisely the potential of the QP-mutation $\mu_i(T_\ell,W_\ell)$. This readily implies the non-degeneracy of $(T_\ell,W_\ell)$. Using techniques of Galois coverings, it is shown in \cite{G-LF-S} that the Jacobian algebra $\jacob{T_\ell}{W_\ell}$ is wild for $\ell=1,2$.
\end{proof}

The next result says that the quivers of the form $\qtau$ are pretty much the only quivers for which we can find non-degenerate potentials with tame Jacobian algebras. That is, if we take an arbitrary quiver $Q$, not necessarily arising from a triangulation of a surface, such that $\jacob{Q}{S}$ is a tame algebra for some non-degenerate potential $S\in\completeRQ$, then we can be almost certain that $Q$ arises from a triangulation of a surface with marked points. Here is the precise statement:

\begin{theorem}\cite{G-LF-S}\label{thm:tame-type-classification} Let $Q$ be any 2-acyclic quiver. If there exists a non-degenerate potential $S\in\completeRQ$ such that the Jacobian algebra $\jacob{Q}{S}$ is tame, then either $Q$ is the quiver associated to a triangulation of a surface with marked points, or $Q$ belongs to the quiver mutation class of one of the following nine exceptional quivers: $E_6$, $E_7$, $E_8$, $\widetilde{E}_6$, $\widetilde{E}_7$, $\widetilde{E}_8$, $E^{(1,1)}_6$, $E^{(1,1)}_7$, $E^{(1,1)}_8$.
\end{theorem}

\begin{proof}[On the proof] To prove Theorem \ref{thm:tame-type-classification} we make use of Theorem \ref{thm:tame-jacobian-algebras} Part (1), Drozd's famous tame/wild dichotomy theorem \cite{Drozd}, and Felikson-Shapiro-Tumarkin's crucial classification of quivers of finite quiver mutation class \cite{FeST}.
\end{proof}

\begin{remark} The quivers $E_6$, $E_7$, $E_8$, $\widetilde{E}_6$, $\widetilde{E}_7$, $\widetilde{E}_8$, $E^{(1,1)}_6$, $E^{(1,1)}_7$ and $E^{(1,1)}_8$ are not very complicated and can be found in \cite{FeST} or \cite{G-LF-S}.
\end{remark}

\section{Uniqueness of non-degenerate potentials}\label{sec:uniqueness}

For quivers of the form $\qtau$ with $\tau$ a (tagged) triangulation of $\surf$, \cite{G-LF-S} classifies all non-degenerate potentials on $\qtau$ in practically all cases. It turns out that almost all quivers of the form $\qtau$ admit exactly one non-degenerate potential. We give the precise statements in Theorems \ref{thm:non-empty-bound-uniqueness} and \ref{thm:empty-bound-uniqueness} below.

\begin{theorem}\cite{G-LF-S}\label{thm:non-empty-bound-uniqueness} Let $\surf$ be a surface with non-empty boundary and any number of punctures, and let $\tau$ be any tagged triangulation of $\surf$.
\begin{enumerate}\item If $\surf$ is not the unpunctured torus with exactly one marked point, then, up to right-equivalence, $\stau$ is the only non-degenerate potential on $\qtau$.
\item If $\surf$ is the unpunctured torus with exactly one marked point, then $\qtau$ admits exactly two non-degenerate potentials up to right-equivalence, namely $\stau$ and another potential $W$ such that $\jacob{\qtau}{W}$ is wild ($W$ is the potential referred to in Theorem \ref{thm:always-tame}).
\end{enumerate}
\end{theorem}

\begin{proof}[On the proof] Given an arbitrary 2-acyclic quiver $Q$, \cite[Theorem 8.19]{G-LF-S} gives a sufficiency criterion for a potential on $Q$ to be the only non-degenerate potential on $Q$. For $\surf$ with non-empty boundary and different from the unpunctured torus with exactly one marked point, in the proof of \cite[Theorem 8.20]{G-LF-S} a triangulation $\sigma$ is constructed whose associated potential $\ssigma$ satisfies the hypothesis of the alluded criterion. Part (1) then follows from Theorem \ref{thm:tagged-flip<->mutation} and the fact that right-equivalent QPs have right-equivalent QP-mutations (the latter is a result of Derksen-Weyman-Zelevinsky).

In the case of the unpunctured torus with exactly one marked point, we already know that $(\qtau,\stau)$ and $(\qtau,W)$ are non-degenerate, where $W=W_2$ is the potential given in \ref{eq:wild-pot-for-torus-w-bound}. That they are not right-equivalent can be proved either directly, or by means of the following argument: We know that $\mathcal{P}(\qtau,\stau)$ is tame and $\mathcal{P}(\qtau,W)$ is wild, and a result of Derksen-Weyman-Zelevinsky tells us that right-equivalent QPs have isomorphic Jacobian algebras; thus $(\qtau,\stau)$ and $(\qtau,W)$ cannot be right-equivalent by Drozd's tame/wild dichotomy theorem.

Given an arbitrary 2-acyclic quiver $Q$, \cite[Lemma 8.18]{G-LF-S} gives a sufficiency criterion for a potential $S\in\complete{Q}$ and a positive integer $m$ to satisfy the property that for every potential $S'\in\maxid^{m+1}$ the QP $(Q,S+S')$ be right-equivalent to $(Q,S)$. Direct computation shows that $\stau$ satisfies the criterion with $m=3$, and $W=W_2$ satisfies the criterion with $m=4$. Finally, using basic linear algebra, it is shown that every non-degenerate potential on $\qtau$ is right-equivalent to $\stau+S'$ or to $W+S'$ for some potential $S'\in\maxid^5$. Part (2) follows.
\end{proof}

\begin{theorem}\label{thm:empty-bound-uniqueness} Let $\surf$ be a surface with empty boundary, and let $\tau$ be any tagged triangulation of $\surf$.
\begin{enumerate}
\item \cite{G-LF-S} If the genus of $\Sigma$ is positive and the number $|\marked|$ of punctures is at least three, then any non-degenerate potential on $\qtau$ is right-equivalent to a non-zero scalar multiple of $\stau$.
\item \cite{Ladkani0} If $\surf$ is a positive-genus surface with exactly one puncture, then the degree-3 component of $\stau$ (that is, the part of $\stau$ that arises from the triangles of $\tau$) is a non-degenerate potential which is not right-equivalent to any scalar multiple of $\stau$.
\item \cite{G-LF-S} If $\surf$ is a sphere with at least five punctures, then any non-degenerate potential on $\qtau$ is right-equivalent to a non-zero scalar multiple of $\stau$.
\end{enumerate}
\end{theorem}

\begin{proof}[On the proof]
Part (1) of Theorem \ref{thm:empty-bound-uniqueness} follows from a combination of several facts (the proofs of some of which are rather technical):
\begin{itemize}
\item[(a)] Any surface as in (1) admits a triangulation $\sigma$ satisfying that no arc in $\sigma$ is a loop, that $\qsigma$ does not have double arrows\footnote{We believe that the absence of double arrows can be deduced from the absence of loops}, and that each puncture is an endpoint of at least four arcs of $\sigma$.
\item[(b)] For any triangulation $\sigma$ as in (a), any potential $S'$ not involving any cycle cyclically-equivalent to a term appearing $\ssigma$, and any collection $\bx=(x_p)_{p\in\punct}$ of non-zero scalars indexed by the punctures of $\surf$, the QP $(\qsigma,S(\sigma,\bx)+S')$ is right-equivalent to $(\qsigma,S(\sigma,\bx))$, where $S(\sigma,\bx)$ is the potential obtained from $\ssigma$ by multiplying each cycle surrounding a puncture $p$ by $x_p$. (The proof of this fact is rather technical and somewhat delicate: the right-equivalence between $(\qsigma,S(\sigma,\bx)+S')$ and $(\qsigma,S(\sigma,\bx))$ given in \cite{G-LF-S} is defined as the composition of three algebra automorphisms of $\complete{\qsigma}$ that are not given explicitly, but rather obtained as limits of certain sequences of automorphisms. The convergence of these sequences is a delicate issue.)
\item[(c)] For any triangulation $\sigma$ as in (a), every non-degenerate potential on $\qsigma$ is right-equivalent to $S(\sigma,\bx)+S'$ for some $S'$ and $\bx=(x_p)_{p\in\punct}$ as in (b).
\item[(d)] For any triangulation as in (a) and any collection $\bx=(x_{p})_{p\in\punct}$ as in (b), the QP $(\qsigma,S(\sigma,\bx))$ is right-equivalent to $(\qsigma,\lambda S(\sigma))$ for some non-zero scalar $\lambda\in\ka$. (For this we use the fact that $\ka$ is algebraically closed: there is an $n$ for which the ground field must have $n^{\operatorname{th}}$ roots of all its elements.)
\item[(e)] For arbitrary potentials $W$ and $W'$ on a given quiver, and any non-zero scalar $\lambda$, if $(Q,W)$ is right-equivalent to $(Q,\lambda W')$, then the QP-mutation $\mu_i(Q,W)$ is right-equivalent to $(\overline{Q},\lambda\overline{W'})$, where $(\overline{Q},\overline{W'})=\mu_i(Q,W')$.
\end{itemize}

For the proof of Part (2) of Theorem \ref{thm:empty-bound-uniqueness}, Ladkani notes that, for once-punctured surfaces with empty boundary and positive genus, the proof of the first statement of Theorem \ref{thm:ideal-flip->QP-mut}, given in \cite{L1} for the QPs of the form $(\qtau,\stau)$ with $\tau$ ideal triangulation, can be easily adapted to show the following: If $\tau$ and $\sigma$ are ideal triangulations of a once-punctured surface with empty boundary and positive genus, and $\tau$ and $\sigma$ are related by the flip of an arc $i$, then $\mu_i(\qtau,\stau^{(3)})$ is right-equivalent to $(\qsigma,\ssigma^{(3)})$, where $\stau^{(3)}$ (resp. $\ssigma^{(3)}$) denotes the degree-3 component of $\stau$ (resp. $\ssigma$). This fact has the non-degeneracy of $(\qtau,\stau^{(3)})$ as a straightforward consequence. That $(\qtau,\stau^{(3)})$ is not right-equivalent to $(\qtau,\lambda\stau)$ for any $\lambda\in\mathbb{C}^*$ follows from the fact that the former QP has infinite-dimensional Jacobian algebra, whereas, due to Ladkani's result Theorem \ref{thm:empty-bound-fin-dim}, the latter QP has finite-dimensional Jacobian algebra.

The proof of part (3) is identical to the proof of part (1) if $\surf$ is a sphere with at least six punctures. The sphere with five punctures is treated separately, but with an argument which is similar to the proof part (1).
\end{proof}

For positive-genus surfaces with empty boundary and exactly two punctures we have the following:

\begin{conjecture}\cite{G-LF-S} If $\tau$ is a tagged triangulation of a positive-genus surface $\surf$ with empty boundary and exactly two punctures, then any non-degenerate potential on $\qtau$ is right-equivalent to a non-zero scalar multiple of $\stau$.
\end{conjecture}

The proof we have sketched of part (1) of Theorem \ref{thm:empty-bound-uniqueness} cannot be applied to prove this conjecture, since the surfaces in the conjecture do not have triangulations without loops.

For the once-punctured torus, we have the following result by Geuenich:

\begin{theorem}\cite{Geuenich}\label{thm:infinitely-many-pots-for-torus} Let $\tau$ be a triangulation of the once-punctured torus. Then:
 \begin{enumerate}\item An arbitrary potential $S\in\qtau$ is non-degenerate if and only if, up to a change of arrows, the degree-3 component of $S$ is equal to either $a_1b_1c_1+a_2b_2c_2$ or $a_1b_1c_2+a_1b_2c_1+a_2b_1c_1$.
 \item There exists an infinite sequence $(S_n)_{n\geq 0}$ of non-degenerate potentials on $\qtau$, with the property that $\lim_{n\to\infty}\dim_{\mathbb{C}}(\jacob{\qtau}{S_n})=\infty$. Hence this sequence has a subsequence $(S_{n_m})_{m\geq0}$ such that for $m_1\neq m_2$, the QP $(\qtau,S_{n_{m_1}})$ is not right-equivalent to $(\qtau,\lambda S_{n_{m_2}})$ for any non-zero scalar $\lambda$.
 \end{enumerate}
\end{theorem}

We expect that Part (2) of Theorem \ref{thm:infinitely-many-pots-for-torus} can be generalized to any once-punctured surface without boundary.

\section{Some applications}\label{sec:applications}

Quivers with potentials can be thought of as a tool to obtain categories. Besides the module categories of Jacobian algebras, there are other categories associated to QPs that are of interest not only to representation-theorists, but to authors from other areas as well. We give a very rough description of a couple of these categories, the reader is referred to \cite{Amiot-gldim2} and \cite[Section 3]{Amiot-survey} for precise definitions, statements and citations.

Let $(Q,S)$ be a non-degenerate QP, and let $\overline{Q}$ be the graded quiver whose vertex set is $Q_0$ and whose arrows are:
\begin{itemize}
\item[In degree 0:] All arrows of $Q$;
\item[In degree -1:] an arrow $a^*:i\to j$ for each arrow $a:j\to i$ of $Q$;
\item[In degree -2:] an arrow $t_i:i\to i$ for each vertex $i\in Q_0$.
\end{itemize}
Every path on $\overline{Q}$ has non-positive degree (defined as the sum of the degrees of its constituent arrows), and for $\ell\in\mathbb{Z}$ an element $u\in\complete{\overline{Q}}$ is said to have degree $\ell$ if all its constituent paths have degree $\ell$. Letting $\widehat{\Gamma}^{(\ell)}$ denote the set of all degree-$\ell$ elements of $\complete{\overline{Q}}$, we see that $\widehat{\Gamma}^{(\ell)}$ consists of all possibly infinite $\ka$-linear combinations of paths of length $\ell$, that
\begin{equation}\label{eq:Ginzburg-grading}
\complete{\overline{Q}}=\prod_{\ell\in\mathbb{Z}}\widehat{\Gamma}^{(\ell)}
\end{equation}
as vector spaces, and that $\widehat{\Gamma}^{(0)}=\complete{Q}$. There is a differential $d$ on $\complete{\overline{Q}}$ that can be defined as the unique degree-$1$ continuous $\ka$-linear map that satisfies the rules
$$
d(a)=0 \ \ \  \text{and} \ \ \  d(a^*)=\partial_a(S) \ \ \ \text{for all $a\in Q_1$},
$$
$$
d(t_i)=e_i\left(\sum_{a\in Q_1}[a,a^*]\right)e_i \ \ \ \text{for all $i\in Q_0$},
$$
together with the Leibniz rule $d(uv)=(du)v+(-1)^\ell udv$ for all homogeneous $u\in \widehat{\Gamma}^{(\ell)}$ and $v$.
The complete path algebra $\complete{\overline{Q}}$, together with the grading \eqref{eq:Ginzburg-grading} and the differential $d$ above, is called the \emph{complete Ginzburg DG algebra}\footnote{\emph{DG} for \emph{differential graded}} of $(Q,S)$, denoted $\Ginz{Q}{S}$.

Via the \emph{derived category} of the category of DG modules over $\Ginz{Q}{S}$, one arrives at a \emph{3-Calabi-Yau triangulated} category $\mathcal{D}\Ginz{Q}{S}$. When the Jacobian algebra $\jacob{Q}{S}$ is finite-dimensional, Amiot defines the \emph{generalized cluster category} $\mathcal{C}(Q,S)$ as certain quotient of a subcategory of $\mathcal{D}\Ginz{Q}{S}$. The generalized cluster category turns out to be \emph{$\Hom$-finite 2-Calabi-Yau triangulated}. Amiot and Keller-Yang show that QP-mutations induce equivalences of categories:

\begin{theorem}\label{thm:categories-from-QPs} Let $(Q,S)$ be a non-degenerate QP and $i\in Q_0$. Then
\begin{enumerate}
\item $(Q,S)$ induces a canonical \emph{$\operatorname{t}$-structure} on $\mathcal{D}\Ginz{Q}{S}$ whose heart $\mathcal{A}_{(Q,S)}$ is equivalent to the module category of $\jacob{Q}{S}$;
\item\cite{KY} there are two canonical equivalences of triangulated categories $\Phi_{i}^+,\Phi_i^{-}:\mathcal{D}\widehat{\Gamma}(\mu_i(Q,S))\rightarrow\mathcal{D}\Ginz{Q}{S}$, such that the images of the $\operatorname{t}$-structure induced by $\mu_i(Q,S)$ on $\mathcal{D}\widehat{\Gamma}(\mu_i(Q,S))$ are the \emph{right} and \emph{left tilts} of the $\operatorname{t}$-structure induced by $(Q,S)$ on $\mathcal{D}\Ginz{Q}{S}$;
\item\cite{Amiot-gldim2,Amiot-survey} there is a \emph{cluster-tilting} object $T_{(Q,S)}$ of $\mathcal{C}(Q,S)$ canonically attached to $(Q,S)$;
\item\cite{KY,Amiot-gldim2,Amiot-survey} $\mathcal{C}(\mu_i(Q,S))$ is equivalent to $\mathcal{C}(Q,S)$ by means of an equivalence of triangulated categories that sends $T_{\mu_i(Q,S)}$ to the cluster-tilting object of $\mathcal{C}(Q,S)$ obtained from $T_{(Q,S)}$ by \emph{IY-mutation}\footnote{\emph{IY} after \emph{Iyama-Yoshino}} with respect to $i$.
\end{enumerate}
\end{theorem}

The hearts of $\mathcal{D}\Ginz{Q}{S}$ that can be obtained from $\mathcal{A}_{(Q,S)}$ by finite sequences of tilts are often called \emph{canonical hearts}, whereas the cluster-tilting objects of $\mathcal{C}_{(Q,S)}$ that can be obtained from $T_{(Q,S)}$ by finite-sequences of IY-mutations are called \emph{reachable} cluster-tilting objects.

\begin{remark}\label{thm:keller-yang-technicality} Since mutations of QPs are only defined up to right-equivalence, Parts (2) and (4) of Theorem \ref{thm:categories-from-QPs} make  an implicit use of the following fact: up to an equivalence of triangulated categories that takes $\operatorname{t}$-structures to $\operatorname{t}$-structures and tilts to tilts (resp. cluster-tilting objects to cluster-tilting objects and IY-mutations to IY-mutations), the category $\mathcal{D}\widehat{\Gamma}(Q,S)$ (resp. $\mathcal{C}(Q,S)$) does not change if we replace $(Q,S)$ with a QP which is right-equivalent to it. Actually, something stronger is true: $\mathcal{D}\widehat{\Gamma}(Q,S)$ (resp. $\mathcal{C}(Q,S)$) does not change if we replace $(Q,S)$ with a QP which is right-equivalent to $(Q,\lambda S)$ for some non-zero scalar $\lambda$.
\end{remark}

By Theorems \ref{thm:non-empty-bound-fin-dim} and \ref{thm:empty-bound-fin-dim}, every Jacobian algebra of the form $\jacob{\qtau}{\stau}$ is finite-dimensional. Thus, a combination of Theorems \ref{thm:tagged-flip<->mutation} and \ref{thm:categories-from-QPs} yields:

\begin{theorem}\label{thm:categories-from-surfaces} Let $\surf$ be a surface with marked points. If $\surf$ is not a sphere with less than 5 punctures, then there exist:
\begin{enumerate}
\item a 3-Calabi-Yau triangulated category $\mathcal{D}\surf$, with canonical hearts and tilts of canonical hearts combinatorially interpreted as tagged triangulations and flips of tagged triangulations, respectively.
\item A $\Hom$-finite 2-Calabi-Yau triangulated category $\mathcal{C}\surf$, with reachable cluster-tilting objects and IY-mutations of reachable cluster-tilting objects combinatorially interpreted as tagged triangulations and flips of tagged triangulations, respectively\footnote{When $\surf$ is a once-punctured surface with empty boundary, this statement needs a slight refinement.}; cf. \cite[Section 3.4]{Amiot-survey}, \cite[Theorem 4.10]{CI-LF}.
\end{enumerate}
\end{theorem}

Indeed, one defines $\mathcal{D}\surf=\mathcal{D}\Ginz{\qtau}{\stau}$ and $\mathcal{C}\surf=\mathcal{C}(\qtau,\stau)$ for any tagged triangulation $\tau$ of $\surf$. Up to equivalences of triangulated categories, $\mathcal{D}\surf=\mathcal{D}\Ginz{\qtau}{\stau}$ and $\mathcal{C}\surf=\mathcal{C}(\qtau,\stau)$ are independent of $\tau$ by Theorems \ref{thm:tagged-flip<->mutation} and \ref{thm:categories-from-QPs} (see also Remark \ref{thm:keller-yang-technicality}).

Recent work \cite{Bridgeland-Smith} of Bridgeland-Smith shows that spaces of \emph{Bridgeland stability conditions} on the categories $\mathcal{D}\surf$ can be realized as spaces of \emph{quadratic differentials} on the Riemann surface $\Sigma$. When $\Sigma$ has empty boundary, Smith \cite{Smith} has furthermore shown that $\mathcal{D}\surf$ can be interpreted as the \emph{Fukaya category} of a symplectic 6-manifold underlying certain Calabi-Yau-3 varieties that fiber over the surface $\Sigma$.

In physics, 
the QPs $\qstau$, as well as their QP-mutation compatibility with flips, have been used by
Alim-Cecotti-Cordova-Espahbodi-Rastogi-Vafa \cite{ACCERV1}, \cite{ACCERV2}, and Cecotti \cite{Cecotti}, in their study of \emph{$N=2$ quantum field theories} and associated \emph{BPS quivers and spectra}.

Let us give a (extremely rough) sketch of the passage from quadratic differentials to stability conditions when $\partial\Sigma=\varnothing$. We start with how quadratic differentials gives rise to a well-defined triangulated category. A \emph{generic} quadratic differential $\phi$ which is holomorphic on $\Sigma\setminus\marked$ and has poles of order 2 at each $p\in \marked$, induces a \emph{horizontal foliation} on $\Sigma$. A typical curve in this foliation joins either two poles or a pole and a zero of $\phi$. On the set of curves of the foliation that join poles of $\phi$, one can define an equivalence relation by setting two curves to be equivalent if they are parallel. By picking a system of representatives for this equivalence relation we obtain a triangulation $\tau_\phi$ of $\surf$. This triangulation $\tau_\phi$ gives rise to the category $\mathcal{D}\surf$ via its associated QP $(Q(\tau_\phi),S(\tau_\phi))$.

When we said ``generic" in the previous paragraph, we meant that the behavior of $\phi$ described takes place in an open subset of the space of meromorphic quadratic differentials on $\Sigma$ that have poles of order two precisely at the points of $\marked$. The space of generic quadratic differentials has several connected components, called \emph{chambers}. If we allow $\phi$ to vary inside a given chamber, the associated triangulation $\tau_\phi$ does not change. However if we allow $\phi$ to move from one chamber to another, the triangulation $\tau_\phi$ changes. The change suffered by $\tau_\phi$ is either a flip or a \emph{pop}. If it is a flip, by Theorems \ref{thm:tagged-flip<->mutation} and \ref{thm:categories-from-QPs} we see that the category $\mathcal{D}\surf$ does not change (but undergoes a tilt of canonical $\operatorname{t}$-structure). If it is a pop, the situation becomes somewhat subtle, but again, the category $\mathcal{D}\surf$ does not change.

So, when we pick a generic quadratic differential $\phi$, we are picking a $\operatorname{t}$-structure on $\mathcal{D}\surf$. The simples of the heart of this $\operatorname{t}$-structure are in one-to-one correspondence with the arcs in $\tau_\phi$. Each of these arcs is transverse to a curve that joins two zeros of $\phi$. We integrate $\sqrt{\phi}$ along each of these curves, thus obtaining an assignment of a complex number to each arc in $\tau_\phi$. This assignment extends uniquely to a group homomorphism from the Grothendieck group $K_{0}(\mathcal{D}\surf)$ to $\ka$. This group homomorphism is actually a stability condition on $\mathcal{D}\surf$.

We warn the reader that things are very far from being as simple as we have just described (for example, $\phi$ above is not any quadratic differential, but has to be  a \emph{complete and saddle-free GMN differential}). In reality, numerous non-trivial considerations are needed and the picture is a lot more complex (and beautiful) than we have made it seem here, see \cite{Bridgeland-Smith}.

We end the paper with a few remarks regarding the uniqueness of the categories $\mathcal{C}\surf$ and $\mathcal{D}\surf$.

\begin{remark} The fact that the category $\mathcal{C}\surf=\mathcal{C}(\qtau,\stau)$ is independent of $\tau$ does not mean that there cannot exist other generalized cluster categories, defined through other potentials on the same quivers $\qtau$, whose reachable cluster-tilting objects can be parameterized by tagged triangulations, with IY-mutations interpreted as flips of tagged triangulations. In other words, the sole fact that $\mathcal{C}\surf$ is well-defined does not imply that it is the unique generalized cluster category one can associate to a surface. The same comment goes for $\mathcal{D}\surf$.
\end{remark}

\begin{example} Suppose $\surf$ is a torus with exactly one boundary component and exactly one marked point. For each triangulation $\tau$ of $\surf$, let $W(\tau)$ be the potential given by \eqref{eq:wild-pot-for-torus-w-bound}. Then $\jacob{\qtau}{W(\tau)}$ is finite-dimensional, and for any two triangulations $\tau$ and $\sigma$ related by the flip of an arc $i$, the QPs $\mu_i(\qtau,W(\tau))$ and $(\qsigma,W(\sigma))$ are right-equivalent. Hence, the categories $\mathcal{D}_{\operatorname{wild}}=\mathcal{D}\Ginz{\qtau}{W(\tau)}$ and $\mathcal{C}_{\operatorname{wild}}=\mathcal{C}(\qtau,W(\tau))$ are independent of $\tau$, and satisfy the assertions made in Theorem \ref{thm:categories-from-surfaces}. However, since $\jacob{\qtau}{W(\tau)}$ is wild and $\jacob{\qtau}{\stau}$ is tame, the categories $\mathcal{C}_{\operatorname{wild}}$ and $\mathcal{C}\surf$ are most likely not equivalent.
\end{example}

\begin{remark} Let $\surf$ be a once-punctured surface with empty boundary and positive genus. The potentials $S(\tau)^{(3)}$ from Part (2) of Theorem \ref{thm:empty-bound-uniqueness} also satisfy that triangulations related  by a flip have QPs related by QP-mutation. Thus they also give rise to categories $\mathcal{D}\Ginz{\qtau}{S(\tau)^{(3)}}$ that are actually independent from $\tau$. However, since the Jacobian algebras $\jacob{\qtau}{S(\tau)^{(3)}}$ are infinite-dimensional, in order to obtain generalized cluster categories one has to apply Plamondon's construction \cite{Plamondon} rather than that of Amiot. In the case of the once-punctured torus, these comments also apply for the potential given by \eqref{eq:wild-pot-for-torus-without-bound}.
\end{remark}

\section*{Acknowledgements}

The first version of the present survey article was mainly based on a talk contributed by the author to the \emph{Taller de Vinculaci\'{o}n Matem\'{a}ticos Mexicanos J\'{o}venes en el Mundo}, held at the \emph{Centro de Investigaci\'{o}n en Matem\'{a}ticas} (CIMAT), Guanajuato, M\'{e}xico, in August 2012. I am grateful to the organizers of this meeting, No\'{e} B\'{a}rcenas-Torres, Fernando Galaz-Garc\'{i}a and M\'{o}nica Moreno-Rocha, for the opportunity of presenting my work.

The second version of this survey incorporates some recent results that did not appear in the first version. Quite a few of these results were obtained in joint work with Christof Geiss and Jan Schr\"oer, to whom I am grateful for many interesting discussions. I owe thanks to Jan Schr\"oer for the enthusiastic working atmosphere while I was a member of his research group as a postdoc at the Mathematisches Institut of Universit\"at Bonn.

\bibliographystyle{amsalpha}

\end{document}